\documentclass[10pt]{amsart}
\begin{document}
\title {Closures of K-orbits in the flag variety for $U(p,q)$}
\author {William M. McGovern}
\begin{abstract}
We classify the $GL_p\times GL_q$-orbits in the flag variety for
$GL_{p+q}$ with rationally smooth closure, showing that they are all either
already closed or are pullbacks from orbits with smooth closure in a
partial flag variety.
\end{abstract}
\maketitle

\section{Introduction}
Let $G$ be a complex reductive group with Borel subgroup $B$.  The
question of which Schubert varieties in the flag variety $G/B$ are
smooth has received a great deal of attention, particularly in recent
years \cite{BL,BP}.  Less well studied, but very important
for representation theory, are the closures of orbits in $G/B$ under
the action of the fixed point subgroup $K:=G^{\theta}$ of $G$, where
$\theta$ is an involutive automorphism of $G$ \cite{LV}.  Such
orbit closures have been called symmetric varieties by Springer and
are studied by him in \cite{Sp}.  In this paper we use his
techniques to decide which symmetric varieties are smooth in the
special case $G=GL(p+q,{\mathbb C}), K=GL(p,{\mathbb C})\times
GL(q,{\mathbb C})$.  We will give a pattern avoidance criterion for
rational smoothness, along the lines of the well-known one for
rational smoothness of Schubert varieties in type $A$.  We will also
show that all rationally smooth 
symmetric varieties in this case are either closed orbits or pullbacks
of smooth varieties in partial flag varieties and so in particular are
smooth.  In a joint paper with Peter Trapa, we extend the
characterization of the rationally smooth symmetric varieties to the
flag varieties for the real groups $Sp(p,q)$ and $SO^*(2n)$ \cite{MT08}.

I would like to thank Peter Trapa for many very helpful conversations, Leticia Barchini for pointing out an error in an earlier version of this paper, and Ben Wyser for pointing out yet another error.

\section{Preliminaries}
Now let $G=GL(n,{\mathbb C})$, where $n = p+q$, and take $\theta$ to
be conjugation by a diagonal matrix on $G$ with $p$ eigenvalues 1 and
$q$ eigenvalues $-1$,
so that $K=G^{\theta} = GL(p,{\mathbb C})\times
GL(q,{\mathbb C})$.  This group may also be viewed as the
complexification of the maximal compact subgroup $U(p)\times U(q)$ of
the real form $U(p,q)$ of $G$.  Let $B$ be the
subgroup of upper triangular matrices in $G$.  As is well known, the
quotient $G/B$ may be identified with the set of complete flags
$V_0\subset V_1\subset\cdots\subset V_n$ in $\mathbb C^n$. Let $P$ be
the span of the first $p$ vectors of the standard basis of $\mathbb
C^n$.  Recall
that $K$-orbits 
in $G/B$ are parametrized by clans, which are sequences
$\gamma=(c_1,\ldots,c_n)$ of $n$ symbols $c_i$, each either $+$ or $-$
or a natural number, such that every natural number occurs either
exactly twice in $\gamma$ or not at all \cite{MO,Y}.  In this
parametrization the orbit corresponding to $(c_1,\ldots,c_n)$ consists
of flags $V_0\subset\cdots\subset V_n$ for which the dimension of
$V_i\cap P$ equals the number of $+$ signs and pairs of equal numbers
among $c_1,...,c_i$, for all $i$ between 1 and $n$.  In particular,
the number of $+$ signs and pairs of equal numbers in the entire clan
must be
exactly $p$.  We identify two clans if they have the same signs
in the same positions and pairs of equal numbers in the same positions
(so that for example $(1,+,1,-)$ is identified with $(2,+,2,-)$, but
not with $(1,+,-,1)$). We say that the clan $\gamma=(c_1,\ldots,c_n)$
{\sl includes the pattern} $(d_1,\ldots,d_m)$ if there are indices
$i_1<\cdots<i_m$ such that the (possibly shorter) clans
$(c_{i_1},\ldots,c_{i_m})$ and $(d_1,\ldots,d_m)$ are identified.  We
say that $\gamma$ {\sl
  avoids} $(d_1,\ldots,d_m)$ if it does not include it.  If $Q$ is a
parabolic subgroup of $G$, corresponding to an arrangement of the $n$
coordinates into blocks of consecutive coordinates (each block having
only one coordinate if $Q=B$), then the quotient $G/Q$ may be
identified with the set of partial flags $V_0\subset\cdots\subset V_m$
in $\mathbb C^n$ such that the dimension of $V_i$ is the sum of the
sizes of the first $i$ blocks of coordinates.  Then $K$-orbits in
$G/Q$
are likewise parametrized by clans, except that
we identify two clans whenever
corresponding blocks of coordinates have the same signatures (number of $+$ 
signs plus pairs of equal numbers, and similarly for $-$ signs) {\sl
  and} pairs of blocks in one clan have the same number of numbers in
common as the corresponding pairs of blocks in the other clan.
For example,
if $p=q=3$ and $Q$ corresponds to the coordinate arrangement
$(1),(2,3,4,5),(6)$, or to the middle three roots
in the Dynkin diagram of $G$, then the clans $(1,+,2,-,2,1)$ and
$(1,2,2,3,3,1)$ are identified.  If $p=q=4$ and $Q$ corresponds to the
arrangement $(1),(2,3,4),(5,6,7),(8)$ then the clans
$(1,+,1,-,+,2,-,2)$ and $(1,-,+,1,+,-,2,2)$ are identified, but
$(1,+,1,-,+,2,-,2)$ and $(1,+,2,-,+,1,-,2)$ are not, as
the first and second blocks share the number 1 in the first clan,
while the first and third blocks share this number in the second
clan.  

The above notions
of pattern inclusion and avoidance are motivated by the corresponding
ones for permutations in one-line notation, which Lakshmibai and
Sandhya used to characterize rationally smooth Schubert varieties in
type $A$ \cite{LS} and Billey later extended to the other classical types
\cite{B}.  

$K$-orbits in $G/B$ are partially ordered by
containment of their closures.  On the level of clans, this order
includes the following operations: to make an orbit larger, replace a
pair of (not necessarily
adjacent) opposite
signs by a pair of equal numbers; or
interchange a number with a sign so as to move the number farther
away from its equal mate in the clan (and on the same side); or
interchange a pair $a,b$ of
distinct numbers with $a$ to the left of $b$, provided that the mate of
$a$ lies to the left of
the mate of $b$
(\cite[5.12]{RS},\cite[2.4]{Y}).  Thus (the orbit
corresponding to) $(1,+,1,-)$ lies below $(1,2,1,2)$ and $(1,+,-,1)$, while
$(1,2,1,3,2,3)$ lies below $(1,3,1,2,2,3))$ but not below $(1,3,1,3,2,2)$. 
(These operations include but do {\sl not}
coincide with the ones generating the Matsuki-Oshima graph
\cite{MO}.)  In particular, the closed orbits are exactly those
whose clans have only signs, while the open orbit has clan
$(1,2,\dots,q,+\ldots,+,q,q-1,\ldots,1)$, with $p-q$ plus signs, if $p>q$.

We will need a formula for the dimension of the orbit
${\mathcal O}_{\gamma}$ corresponding to the clan
  $\gamma=(c_1,\ldots,c_n)$ \cite[2.3]{Y}.  
Set $d_{p,q}:=\frac{1}{2}(p(p-1) + q(q-1))$.  Then $\dim{\mathcal
  O}_\gamma$ is given by
$$
     d_{p,q}\,\,+\sum_{c_i=c_j\in{\mathbb N},i<j}
    (j-i-\#\{k\in{\mathbb N}: c_s = c_t = k\,\,\mbox{for some}\,
    s<i<t<j\})
$$
In particular, the closed orbits all have the same dimension
$d_{p,q}$.

We conclude this section by recalling the well-known
derived functor module construction on the level of $K$-orbits.  For
this purpose let $G$ be any complex reductive group and $K$ its fixed
points under an involution $\theta$.  Let $Q$ be any $\theta$-stable
parabolic subgroup of $G$, containing the $\theta$-stable Borel
subgroup $B$.  If $\mathfrak q$ is the Lie algebra of $Q$, then the
orbit $K\cdot\mathfrak q$ identifies with a closed orbit in the
partial flag variety $G/Q$ \cite[2.5]{RS}.  Its preimage
$\pi^{-1}(K\cdot{\mathfrak 
  q})$ in $G/B$ under the natural projection $\pi:G/B\rightarrow G/Q$ is the
support of a derived functor module; we call the open orbit in this
preimage a derived functor
orbit \cite[\S1]{Tr}.  Its closure fibers smoothly via $\pi$ over
$K\cdot\mathfrak q$ with fiber the 
flag variety $Q/B$ of $Q$ (which may be identified with the flag
variety of any Levi factor of $Q$), so it is smooth.  The clan of a derived functor orbit is obtained by concatenating the clans of the open orbits in the flag varieties of $\theta$-stable Levi subgroups; a typical example is $(1,2,3,+,+,3,2,1,+,-,4,-,4)$.

\section{Main result}
Now we can characterize the $K$-orbits with rationally smooth closure.

\newtheorem*{result}{Theorem}
\begin{result}
If the clan $\gamma= (c_1,\ldots,c_n)$ includes one of the patterns
$(1,+,-,1),$\linebreak $(1,-,+,1)$,$(1,2,1,2)$, $(1,+,2,2,1),(1,-,2,2,1),(1,2,2,+,1), (1,2,2,-,1)$, or\hfil\linebreak  $(1,2,2,3,3,1)$,  then the orbit ${\mathcal
  O}_{\gamma}$ does not have rationally smooth closure.  Otherwise
${\mathcal O}_{\gamma}$ is a derived functor orbit, so that its
closure is smooth.  In particular, Springer's necessary condition for
rational smoothness in \cite{Sp} is sufficient in this setting and
smoothness and
rational smoothness are equivalent.
\end{result}

\begin{proof}
Suppose first that $\gamma$ includes one of the above patterns. If this pattern has just two equal numbers, replace them by $-$ and $+$, in that order; if it includes two such pairs, replace the four numbers by $-,+,-,+$, in that order; if it includes three such pairs, replace the numbers by $-,+,-,+,-,+$, in that order. 
In all eight cases, continue by replacing every pair
$a,\ldots,a$ of equal numbers in $\gamma$ by $+,\ldots,-$.  We obtain
a clan corresponding to a closed orbit $\mathcal O$ below ${\mathcal
  O}_\gamma$ in the partial order.  Now Springer has defined an action
of the noncompact root reflections in the Weyl group $S_n$ on the
closed orbits, sending each such orbit to a higher orbit whose clan
has exactly two numbers; more precisely, any two opposite signs in the
clan of the closed orbit may be replaced by a pair of equal numbers
\cite[3.1,4.1]{Sp}.  One easily checks that
more than $\dim{\mathcal O}_\gamma - d_{p,q}$ of these reflections
send $\mathcal O$ to an orbit lying between it and ${\mathcal
  O}_\gamma$, whence ${\mathcal O}_\gamma$ is not rationally smooth,
as claimed \cite[3.2,3.3]{Sp}.

Now suppose that $\gamma$ avoids the above patterns.  Then the
intervals $[s,t]$ of indices $s,t$ with $c_s=c_t\in{\mathbb N}$ are
such that any two of them are one contained in the other or disjoint.
All signs lying between any pair of equal numbers in
$\gamma$ are the same.  If a sign lies between a pair of equal numbers, then it also lies between every pair of equal numbers enclosed by the first pair.  Finally, if one pair of equal numbers lies inside another, then the pairs of equal numbers lying inside this pair form a single nested chain.  The orbit must then be a derived functor orbit and so have smooth closure.
\end{proof}

In future work we hope to find similar pattern avoidance criteria for
rational smoothness of $K$-orbit closures in the flag varieties of
other classical groups.  There are two nonsmooth orbit closures for
$GL(4,{\mathbb R})$, none for $SU^*(4)$, and one for $SU^*(6)$.

\end{document}